\documentclass[12pt]{amsart}

\usepackage{amsmath,amsthm,amscd,amsfonts,amssymb,epic,enumerate, eepic,bbm, mathabx, youngtab, ytableau}
\usepackage[pagebackref,colorlinks=true,linkcolor=blue,citecolor=blue]{hyperref}

\allowdisplaybreaks
\newtheorem{thm}{Theorem}[section]

\newtheorem{conj}[thm]{Conjecture}

\newtheorem{prop}[thm]{Proposition}
\newtheorem{quest}[thm]{Question}
\theoremstyle{remark}
\newtheorem*{rem}{Remark}
\newtheorem{example}{Example}

\newcounter{remarkscounter}

\numberwithin{equation}{section}
\newcommand{\A}{\mathbb{A}}
\newcommand{\GL}{\mathrm{GL}}
\newcommand{\SL}{\mathrm{SL}}

\newcommand{\Sym}{\mathrm{Sym}}

\newcommand{\lto}{\longrightarrow}

\newcommand{\CC}{\mathbb{C}}

\newcommand{\la}{\lambda}
\newcommand{\Sla}{\mathbb{S}_{\parallel\la\parallel}}

\newcommand{\quash}[1]{}

\theoremstyle{definition}
\newtheorem{defn}[thm]{Definition}

\numberwithin{equation}{subsection}

\renewcommand{\hat}{\widehat}

\begin{document}
\title[Classical groups with type $B_n$, $C_n$ and $D_{2n}$]{On Classical groups detected by the tensor third representation}
\author{Heekyoung Hahn}
\address{Department of Mathematics\\
Duke University\\
Durham, NC 27708}
\email{hahn@math.duke.edu}



\maketitle
\begin{abstract}
Motivated by the Langlands' beyond endoscopy proposal for establishing functoriality, we study the representation $\otimes^3$ in a setting related to the Langlands $L$-functions $L(s,\pi,\,\otimes^3),$ where $\pi$ is a cuspidal automorphic representation of $G$ where $G$ is either $\mathrm{SO}(2n+1)$, $\mathrm{Sp}(2n)$ and $\mathrm{SO}(2n)$. In particular, under what conditions on partitions $\la$, we examine whether or not  $\otimes^3$ detects the subgroups $\mathbb{S}_{[\la]}(G)$ for $G$ with type $B_n$ and $D_{2n}$ or $\mathbb{S}_{\langle\la\rangle}(G)$ for $G$ with type $C_n$. Here $\mathbb{S}_{[\la]}$ and $\mathbb{S}_{\langle\la\rangle}$ are the usual Schur functors associated to the partition $\la$.
\end{abstract}
\tableofcontents

\section{Introduction}

Let $F$ be a number field and let $\A_F$ be adeles of $F$. Let $H$ be a reductive group over $F$. Given a representation $${}^LH\lto \GL_N(\CC),$$ Langlands' functorial conjectures \cite{Langlands_conj} predict there should be a corresponding transfer of automorphic representations of $H(\A_F)$ to automorphic representations of $\GL_N(\A_F)$.

One can ask for a characterization of those automorphic representations in the image. It boils down to understanding how to detect the $L$-parameters such that a $\GL_N(\CC)$-conjugate factors through ${}^LH$. If we let ${}^{\lambda}H$ be the Zariski closure of $\mathrm{Im}({}^LH)$ viewed as a reductive group over $\CC$, then by a theorem of Chevalley \cite{Milne} one knows that there exists a representation $\GL_N \lto \GL(V)$ such that ${}^{\lambda}H$ is the stabilizer of a line in $V$. Therefore the following definition is natural (see \cite{Hahn}):

\begin{defn}
Let $H$ be an irreducible reductive subgroup of $\GL_N$. We say a representation $r:\GL_N \lto \GL(V)$ \textit{detects} $H$ if $H$ fixes a line in $V$.
\end{defn}

\begin{rem}
If $H$ is connected  then $r$ detects $H$ if and only if it detects $H^{\mathrm{der}}$.
\end{rem}

\noindent The following conjecture is the crux of Langlands' beyond endoscopy proposal \cite{Langlands_beyond}, which aims to prove Langlands functoriality in general:

\begin{conj}\label{conj}
Let $\pi$ be a unitary cuspidal automorphic representation of $\GL_N(\A_F).$ If $\pi$ is a functorial transfer from $H$, then $L(s,\pi,r \otimes \chi)$ has a pole at $s=1$ for some character $\chi \in F^{\times} \backslash \A_F^{\times} \to \CC^\times$ whenever $r$ detects ${}^{\lambda}H$.
\end{conj}

Motivated by Langlands' proposal, in the recent paper \cite{Hahn}, the author proposed the following concrete question in algebraic group theory:
 
\begin{quest}\label{Q}
Given a representation 
$$
r:\GL_N \lto \GL(V)
$$
which algebraic subgroups of $\GL_N$ are detected by $r$?
\end{quest}

If $r=\Sym^2$, one knows that every irreducible reductive subgroup of $\GL_N$ detected by $r$ is conjugate to a subgroup of $\mathrm{GO}_N$. Moreover, in this case the Conjecture \ref{conj} is proven by work of Arthur \cite{Arthur}, work of Cogdell, Kim, Piatetski-Shapiro and Shahidi \cite{CKPSS} and work of Ginzburg, Rallis and Soundry \cite{GRS}.  There is a similar statement for $r=\Lambda^2$.  Thus the case $r=\otimes^2$ is relatively well-understood.

Apart from this special case, explicit results are hard to come by. We mention one case that was discussed in a recent paper of Getz and Klassen \cite{GK}.
Let 
$$
RS: \GL_m\times\GL_m \hookrightarrow \GL_{m^2}
$$
be the representation induced by the usual tensor product. Then it is known that $RS(\GL_m\times\GL_m)$ is detected by $\mathrm{Sym}^m$ (see \cite{GK} for instance). Moreover, the analytic properties of the relevant Langlands  $L$-function is well understood via the Rankin-Selberg theory, although we do not know its automorphicity:

\begin{rem}
Let $\pi_1$ and $\pi_2$ be cuspidal automorphic representations of $\GL_m(\A_F)$.   Then one knows that
$$
L(s,\pi_1 \times \pi_2,RS)=L(s,\pi_1 \times \pi_2).
$$
Here the left is the Langlands $L$-function and the right is the usual Rankin-Selberg $L$-function.
\end{rem}

One then ask, for instance, whether $RS(\GL_m\times\GL_m)$ is the only maximal subgroup of $\GL_{m^2}$ that is detected  by $\mathrm{Sym}^m$. It turns out that is not the case. In fact, the author proved \cite{Hahn} that subgroup $H:=\mathrm{Sym}^{n-1}(\SL_2)$ of $\GL_n$ is detected by $\Sym^3$ if and only if $n\equiv 1\pmod{4}$. Moreover, the author showed that if $m>\ell(\la)$ then the representation $\otimes^3$ does not detect  $\mathbb{S}_{\lambda}(\SL_m)$. Here $\mathbb{S}_{\lambda}$ is the Schur functor associated to a partition $\la$ and $\ell(\la)$ is the number of nonzero parts of $\la$.  

Motivated both by Rankin-Selberg theory and the fact that $\otimes^3$ is some sense the next natural case to consider after $\otimes^2$, we continue to investigate what groups are detected by the representation $\otimes^3$.

For the remainder of the paper we make the following assumption on $G$:
\begin{itemize}
\item[(A1)] The algebraic group $G$ is one of the classical groups $\mathrm{SO}(2n+1)$, $\mathrm{Sp}(2n)$ and $\mathrm{SO}(2n)$, where for $\mathrm{SO}(2n)$, we further assume $n$ is even.  
\end{itemize}
It is well-known \cite[Lemma 4.2]{Seitz} that all irreducible modules of $G$ are then self-dual.  This is false for $\mathrm{SO}(2n)$ when $n$ is odd, which is why we assume that $n$ is even.
We investigate which irreducible subgroups of $\GL_N$ are detected by the representation $\otimes^3: \GL_N \lto \GL_{N^3}$.

Throughout the paper, we use the following notation for a partition $\la$ with at most $n$ parts:
$$\la=(\la_1, \la_2, \dots, \la_n),\quad \la_1\geq\la_2\geq \dots\geq \la_n\geq 0.
$$
Let $\Sla(G)$ be the irreducible subgroup of $\GL_N$ obtained by applying the Schur functor associated to the highest weight $\la$.  The precise construction depends on the type of $G$ and is reviewed in \S \ref{prelim} below. We make the following additional assumption on $\la$:
\begin{itemize} 
\item[(A2)] For $G=\mathrm{SO}(2n)$, we assume that $\lambda_n=0$.
\end{itemize}
We make this assumption to ensure that the representation $\Sla$ of $G$ is irreducible.

If $\la$ is a partition of an odd number, we obtain the following complete answer:

\begin{thm}\label{thm-odd-intro}
Let $\la$ be a partition of an odd number. Then the representation $\otimes^3$ does not detect $\Sla(G)$.
\end{thm}

For a partition $\la$ of an even number, we obtain the following results depending on $\la$. 

\begin{thm}\label{thm-even-intro}
Let $\la$ be a partition of an even number. If $\la$ satisfies
\begin{enumerate}
\item[(a)]{all parts are even, or}
\item[(b)]{it has even $\ell(\la)$ and all parts are distinct and odd, or}
\item[(c)]{$\la$ is a hook partition, or}
\item[(d)]{$\la$ is a rectangular partition,}
\end{enumerate}
then the representation $\otimes^3$ detects $\Sla(G)$.
\end{thm}

Before outlining the paper we comment on these results from the perspective of Langlands' beyond endoscopy proposal.  
Essentially they give some feeling of how much more complicated the structure of the beyond endoscopy proposal is than the theory of endoscopy.  In the theory of endoscopy one writes the trace formula in terms of stable orbital integrals on endoscopic groups.  In the beyond endoscopy proposal Langlands proposes that one writes limiting forms of the trace formula in terms of groups that are detected by a particular representation.  It is evident from the theorems above that this set is much more complicated than the set of endoscopic groups of a given group.  On the other hand there may be simplifications that can be made in certain situations.  For example, all of the groups $\mathbb{S}_{\parallel\lambda\parallel}(G) \leq \GL_N$ considered above are conjugate to subgroups of $\mathrm{O}_N$ or $\mathrm{Sp}_N$ since the representations $\mathbb{S}_{\parallel\lambda\parallel}$ are self-dual.  Thus for some purposes it might be possible to sieve them all out at the outset by restricting to non-self dual representations.

We close our introduction by outlining the paper. In \S \ref{prelim}, we review some basic facts on partitions, the Littlewood-Richardson coefficients and the groups $\mathrm{SO}(2n+1)$, $\mathrm{Sp}(2n)$ and $\mathrm{SO}(2n)$. In \S \ref{third}, we prove a key proposition showing a necessary and sufficient condition for detection by $\otimes^3$.  In \S \ref{odd}, we prove that the representation $\otimes^3$ does not detect $\Sla(G)$ for partitions $\la$ of the odd numbers. In the last section \ref{even}, we prove Theorem \ref{thm-even-intro} for partitions $\la$ of the even numbers.


\section{Preliminaries}\label{prelim}

\subsection{Partitions}

In this section, we recall some basic notion of partitions. Let $\lambda$ be a partition with at most $n$ parts written as
\begin{equation}\label{partition-nparts}
\lambda=(\lambda_1, \lambda_2, \dots , \lambda_n), \quad \lambda_1\geq\lambda_2\geq \cdots\geq \lambda_n\geq 0,
\end{equation}
and let $|\lambda|=\sum_i\la_i$ be the number partitioned by $\lambda$. Moreover, denote by $\ell(\la)$ the number of parts of a partition $\la$.

The Young diagram of the partition $\lambda=(\lambda_1, \lambda_2, \dots , \lambda_n)$ is an array of boxes arranged in left-justified horizontal rows where for each $i$, $\la_i$ is the number of boxes in the $i$th row \cite{Fulton}. For example, the Young diagram of the partition $\la=(4,4,3,1, 1)$ is
$$
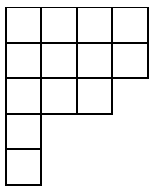
\yng(4,4,3,1,1)
$$

For positive integers $a_i$, $\la_i$ with $1\leq i\leq k$ and $\la_i>\la_{i+1}$ for all $1\leq i\leq k-1$, the notation $\la=(\la_1^{a_1}, \cdots,\la_k^{a_k})$ is shorthand for the partition of $\sum_{i=1}^ka_i\la_i$ given by
$$
\la=(\underbrace{\la_1, \dots, \la_1}_{a_1 \text{ times }}, \dots, \underbrace{\la_k, \dots, \la_k}_{ a_k \text{ times }}).
$$ 
Therefore the above partition $\la=(4,4,3,1, 1)$ can be shorten to $\la=(4^2, 3, 1^2)$. Note that in this example, $|\la|=13$ and $\ell(\la)=5$ which corresponds the number of rows of its Young diagram.

The conjugate of a partition $\la$ is the partition of $|\la|$ whose Young diagram is obtained by reflecting the Young diagram of $\la$ about the diagonal so that rows become columns and columns become rows. We write it as $\la'$. In above example $\la=(4,4,3,1,1)$, the conjugate partition of $\la$ is
$\la'=(5,3,3,2)$ and its Young diagram is
$$
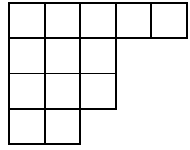
\yng(5,3,3,2)
$$
Many interesting partition identities were proved by using the Young diagram together with notion of conjugate partitions, see for example \cite{Andrews}.

\subsection{Littlewood-Richardson coefficients}

In this section, we follow the exposition in \cite{FH} as we recall basic facts: For any $n$ dimensional vector space $V$ over $\CC$ and any partition  $\lambda$ with at most $n$ parts as in \eqref{partition-nparts}, we can apply the Schur functor $\mathbb{S}_{\lambda}$ to $V$ to obtain a representation $\mathbb{S}_{\lambda}(V)$ for $\GL_n$. It remains irreducible when restrict to $\SL_n$. In particular it determines an irreducible representations of the Lie algebra $\frak{sl}_n$ (see \cite[Proposition 15.15]{FH}).

By the Littlewood-Richardson formula (compare with \cite[Exercise 15.23]{FH}), one knows the decomposition of a tensor product of any two irreducible representations of $\mathfrak{sl}_n$, namely
\begin{equation}\label{LR}
\mathbb{S}_{\lambda}(V) \otimes \mathbb{S}_{\mu}(V)=\bigoplus_{\nu}c^{\nu}_{\lambda\mu}\mathbb{S}_{\nu}(V).
\end{equation}
Here $\nu$ is a partition of $|\lambda|+|\mu|$ and the coefficient $c^{\nu}_{\lambda\mu}$ are given by the Littlewood-Richardson rule. The constant $c^{\nu}_{\lambda\mu}$ is the number of ways to obtain the partition $\nu$ from the partition $\la$ by ``adding" the partition $\mu$ following the Littlewood-Richardson rule.

As our proofs rely heavily on this rule, we will briefly state it by following the exposition in \cite[\S 4]{HL}. Let $\la=(\la_1, \la_2, \dots, \la_n)$ and let $
\mu=(\mu_1, \mu_2, \dots, \mu_n)$ be partitions defined as in \eqref{partition-nparts}.
 One writes $\la\subset \nu$ if the Young diagram of $\la$ sits inside the Young diagram of $\nu$. In other words, $\la_i\leq \nu_i$ for all $i$. If $\la\subset \nu$, put the Young diagram of $\la$ on the Young diagram of $\nu$ with the same top-left corner and remove $\la$ out of $\nu$. That way, we obtain the \textit{skew diagram} $\nu-\la$. Put a positive number in each box $\nu-\la$, then it becomes a \textit{skew tableau} with the \textit{shape} $\nu-\la$. If the entries of this skew tableau are taken from $\{1, 2, \dots, n\}$ and $\mu_j$ of them are $j$ for each $j=1, 2, \dots, m$, then the \textit{content} of this skew tableau becomes $\mu=(\mu_1, \dots, \mu_n)$. For a skew tableau $T$, we define the \textit{word} of $T$ by the sequence $w(T)$ of positive integers obtained by reading the entries of $T$ from top to bottom and right to left in each row.

For example,
\begin{equation}\label{First-LRT}
T=\ytableausetup{centertableaux}
\ytableaushort
{\none\none\none 1 1 1, \none\none 1 2,\none 223,34}
*{6,4,4,2}
*[*(yellow)]{3,2,1}
\end{equation}
is a skew tableau of shape $\nu-\la$, where $\la=(3,2,1)$ and $\nu=(6, 4,4,2)$ and the content is $\mu=(4,3,2,1)$. Its word is 
$$
w(T)=(1,1,1,2,1,3,2,2,4,3).
$$

\begin{defn}\label{LRT}
A Littlewood-Richardson tableau is a skew tableau $T$ with the following properties:
\begin{enumerate}[(i)]
\item{The numbers in each row of $T$ weakly increase from left to right and the numbers in each column of $T$ strictly increase from top to bottom.}
\item{For each positive integer $j$, starting from the first entry of $w(T)$ to any place in $w(T)$, there are at least as many as $j$'s as $(j+1)$'s.}
\end{enumerate}
\end{defn}
\noindent Then the Littlewood-Richardson coefficient $c^{\nu}_{\la\mu}$ is the number of the Littlewood-Richardson tableau of shape $\nu-\la$ and content $\mu$. The skew tableau $T$ in \eqref{First-LRT} is indeed a Littlewood-Richardson tableau.

\subsection{Classical groups $\mathrm{SO}(2n+1)$, $\mathrm{Sp}(2n)$ and $\mathrm{SO}(2n)$} 

In this section, we briefly go over the basic facts on orthogonal and symplectic groups which we will be using mainly for the purpose of fixing notations. We follow the expositions from \cite{FH} in part. 

Let $V$ be a complex vector space equipped with a nondegenerate symplectic or orthogonal form $Q$ such that
\begin{equation}\label{dim}
\dim V=2n+\delta, \quad \delta\in\{0, 1\}.
\end{equation}
Denote by
$$
\{e_1, e_2, \dots, e_{2n+\delta}\}
$$
a basis such that
$$
Q(e_i, e_{2n+1+\delta-i})=\pm Q(e_{2n+1+\delta-i}, e_i)=1
$$
for all $i=1, 2, \dots, n+\delta$ and such that all other pairings are $0$.

Let $G$ be the subgroup of $\SL(V)$ which preserves this form. Then by using the Cartan-Killing classification, one knows that $G$ is of type $B_n$ (resp. type $D_n$) if $Q$ is orthogonal and $\dim V=2n+1$ (resp. $\dim V=2n$). In the former case we write $G=\mathrm{SO}(2n+1)$ and in the latter we write $G=\mathrm{SO}(2n)$.  
If $Q$ is symplectic which forces $\dim V=2n$, we say that $G$ is of type $C_n$
and write $G=\mathrm{Sp}(2n)$.

Given any partition $\la$ with at most $2n+\delta$ parts, one obtains an irreducible representation $\mathbb{S}_{\la}(V)$ of $\SL(V)$ of highest weight $\la$ inside of $V^{\otimes |\la|}$. Given the form $Q$ and integers $1\leq i< j\leq |\la|$, one has a contraction map
$$
V^{\otimes |\la|}\lto V^{\otimes (|\la|-2)}
$$defined by
$$
v_1\otimes v_2\otimes \cdots\otimes v_{|\la|}\mapsto Q(v_i, v_j)v_1\otimes \cdots\otimes \hat{v}_i\otimes \cdots\otimes\hat{v}_j\otimes\cdots\otimes v_{|\la|},
$$where the hat means we have removed those two vectors.
Define $\mathbb{S}_{\langle\la\rangle}(V)$ (resp. $\mathbb{S}_{[\la]}(V)$) to be the intersection of $\mathbb{S}_\la(V)$ with the kernels of all possible contraction maps of symplectic form $Q$ (resp. orthogonal form $Q$). 

Then it is well-known that for type $B_n$ and $C_n$, $\mathbb{S}_{[\la]}(V)$ and $\mathbb{S}_{\langle \la\rangle}(V)$ are the irreducible representation of $G$ with highest weight $\la=(\la_1, \la_2, \dots, \la_{n-1}, \la_n)$. For type $D_n$, the same is true if $\la_n=0$. For type $D_n$ with $\la_n>0$, then $\mathbb{S}_{[\la]}(V)$ is the direct sum of two irreducible representations of $G$, one with highest weight $\la$ and the other with highest weight $\la^{-}=(\la_1, \dots, \la_{n-1}, -\la_n)$. For the proofs, see \cite[Theorem 17.11]{FH} for type $C_n$ and see \cite[Theorem 19.22]{FH} for type $B_n$ and $D_n$, for example.

Analogous to the tensor product decomposition \eqref{LR}, it is known \cite{King, KT} that there is an explicit decomposition formula of the tensor product of two given irreducible representations of each classical group into irreducible constituents. Moreover, Koike and Terada \cite{KT} prove that the decomposition rules for the tensor product coincide completely for $\mathrm{SO}(2n+1)$, $\mathrm{Sp}(2n)$ and $\mathrm{SO}(2n)$. 

For convenience, we denote by $\mathbb{S}_{\parallel\la\parallel}(V)$ either the $\mathbb{S}_{[\lambda]}(V)$ or the $\mathbb{S}_{\langle\lambda\rangle}(V)$ depending on group $G$. There is common structure constant $N^{\nu}_{\lambda\mu}$ such that 
\begin{equation}\label{GenLR}
\Sla(V) \otimes \mathbb{S}_{\parallel\mu\parallel}(V)=\bigoplus_{\substack{\nu\\\ell(\nu)\leq n}}N^{\nu}_{\lambda\mu}\mathbb{S}_{\parallel\nu\parallel}(V),
\end{equation}
where the coefficients $N^{\nu}_{\lambda\mu}$ are given by the Newell-Littlewood formula:
\begin{equation}\label{GenMul}
N^{\nu}_{\lambda\mu}=\sum_{\alpha, \beta, \gamma}c^{\la}_{\alpha \beta}c^{\mu}_{\alpha \gamma}c^{\nu}_{\beta\gamma},
\end{equation}where $c$'s denote the usual Littlewood-Richardson coefficients. The sum is over all partitions $\alpha$, $\beta$ and $\gamma$ (see \cite{SZ} and \cite[(25.27)]{FH}).


\section{Detection via the representation $\otimes^3$}\label{third}

For $G$ considered in current paper, recall that all irreducible modules of such $G$ are self-dual\cite[Lemma 4.2]{Seitz}. In this section, we prove the following key proposition:

\begin{prop}\label{key}
The representation $\otimes^3$ detects $\Sla(G)$ if and only if $N^{\la}_{\la\la} > 0$, where $N^{\la}_{\la\la}$ is the common structure constant as in \eqref{GenMul}.
\end{prop}

\begin{proof}
Let $V$ be the standard representation of $G$. Then as described in \S \ref{prelim}, one obtains irreducible representations $\Sla(V)$ of $G$ with highest weight $\la$.

One has
\begin{align}
\Sla(V)^{\otimes 3}&\cong \mathrm{Hom}(\Sla(V)^{\vee}, \Sla(V)\otimes \Sla(V))\nonumber\\
&\cong\bigoplus_{\substack{\nu\\\ell(\nu)\leq n}}N^{\nu}_{\lambda\la}~ \mathrm{Hom}(\Sla(V)^{\vee},\mathbb{S}_{\parallel\nu\parallel}(V))\nonumber\\
&\cong\bigoplus_{\substack{\nu\\\ell(\nu)\leq n}}N^{\nu}_{\lambda\la}~ \mathrm{Hom}(\Sla(V), \mathbb{S}_{\parallel\nu\parallel}(V)), \label{Schur-Hom}
\end{align}
where we employ \eqref{GenMul} and use the fact $\Sla(V)$ is self-dual (see \cite{Seitz}). Therefore $\Sla(G)$ is detected by $\otimes^3$ if and only if for some $\la$ and $\nu$,
$$\mathrm{Hom}(\Sla(V), \mathbb{S}_{\parallel\nu\parallel}(V))\neq 0\quad\text{and}\quad N^{\nu}_{\la\la}\neq 0.$$

\noindent Notice that two partitions $\la$ and $\la'$ with at most $n$ parts (for $G=\mathrm{SO}(2n)$, again we assume $\la_n=0$) determine the same representation of $G$ if and only if $\la=\la'$. This is due to highest weight theory: the representation $\Sla(V)$ is the irreducible representation of $G$ with the highest weight $\la$ (see \cite[Theorem 17.11]{FH} for the symplectic group $G$ and \cite[Theorem 19.22]{FH} for the orthogonal group $G$). Thus $\Sla(G)$ is detected by $\otimes^3$ if and only if
\begin{equation}\label{main-Hom}
\mathrm{Hom}_G(\Sla(V), \mathbb{S}_{\parallel\nu\parallel}(V))\neq 0\quad\text{and}\quad N^{\nu}_{\la\la}\neq 0.
\end{equation}
In fact, this implies that $N^{\la}_{\la\la}>0$ if and only if $\Sla(G)$ is detected by $\otimes^3$. 
\end{proof}

Therefore in the rest of the paper, by employing Proposition \ref{key}, we figure out conditions on partitions $\la$ on which $\Sla(G)$ is detected by $\otimes^3$.


\section{Subgroups not detected by $\otimes^3$}\label{odd}

By Proposition \ref{key}, detection boils down to check whether the coefficient $N^{\la}_{\lambda\la}$ in \eqref{GenMul} is zero or not. For odd $|\la|$, we obtain a complete answer:

\begin{thm}\label{thm-odd}
Let $\la$ be a partition of an odd number. Then the representation $\otimes^3$ does not detect  $\Sla(G)$.
\end{thm}

\begin{proof}
The proof is based on a simple fact that the coefficient $N^{\nu}_{\lambda\mu}$ is zero unless $|\la|+|\mu|+|\nu|$ is even. This is because $c^{\la}_{\alpha \beta}=0$ unless $|\alpha|+|\beta|=|\la|$. Therefore it is clear that $N^{\la}_{\la\la}=0$ unless $|\la|$ is even. Hence if $\la$ is a partition with odd $|\la|$, then the representation $\otimes^3$ cannot detect $\Sla(G)$.
\end{proof}

\section{Subgroups detected by $\otimes^3$}\label{even}

In this section, we will deal with the remaining case, namely the partitions $\la$ such that $|\la|$ is even. We will rewrite the statement of the result  in this case:

\begin{thm}\label{thm-even}
Let $\la$ be a partition of an even number. If $\la$ satisfies
\begin{enumerate}
\item[(1)]{all parts are even, or}
\item[(2)]{it has even $\ell(\la)$ and all parts are distinct and odd, or}
\item[(3)]{$\la$ is a hook partition, or}
\item[(4)]{$\la$ is a rectangular partition,}
\end{enumerate}
then the representation $\otimes^3$ detects $\Sla(G)$.
\end{thm}

\begin{proof}
The proof follows from Propositions \ref{prop-even}, \ref{prop-dist}, \ref{prop-hook} and \ref{prop-rec}, respectively.
\end{proof}
\noindent We will consider all cases separately in following subsections. 

\subsection{All parts are even}

In this section, we consider the first case of Theorem \ref{thm-even}.

\begin{prop}\label{prop-even}
Let $\la$ be a partition such that all parts are even. Then the representation $\otimes^3$ detects $\Sla(G)$.
\end{prop}

\begin{proof}
Let $\la=(\la_1, \la_2, \dots, \la_n)$ be the partition such that all $\la_i$ are even. From the fact \eqref{Schur-Hom}, one needs to prove that the coefficient $N^{\la}_{\la\la}$ given by
$$
N^{\la}_{\lambda\la}=\sum_{\alpha, \beta, \gamma}c^{\la}_{\alpha \beta}c^{\la}_{\alpha \gamma}c^{\la}_{\beta\gamma}
$$
is positive.

We will construct partitions $\alpha$, $\beta$ and $\gamma$ such that $c^{\la}_{\alpha \beta}\geq 1$, $c^{\la}_{\alpha \gamma}\geq 1$ and $c^{\la}_{\beta\gamma}\geq 1$, thus in turn $N^{\la}_{\la\la}\geq 1$. Let $\alpha$, $\beta$ and $\gamma$ be the partitions give by
\begin{equation}
\alpha=\beta=\gamma=(\tfrac{\la_1}{2}, \tfrac{\la_2}{2}, \dots, \tfrac{\la_n}{2}).
\end{equation}
Then it is not hard to obtain that, for example, a skew tableau of shape $\la-\alpha$ with content $\beta$ in an obvious way: Put $1$'s into $\tfrac{\la_1}{2}$ boxes in the first row of $\beta$ and add them to the right side of the $\tfrac{\la_1}{2}$ empty boxes in the first row of $\alpha$. This becomes then the first row of $\la$  with $\la_1$ boxes where the first half of them are empty and the remaining half of them have $1$'s in them. Likewise, put $2$'s into $\tfrac{\la_2}{2}$ boxes in the second row of $\beta$ and add them to the right side of the $\tfrac{\la_2}{2}$ empty boxes in the second row of $\alpha$. Again then this becomes the second row of $\la$  with $\la_2$ boxes where the first half of them are empty and the remaining half of boxes have $2$'s in them. We repeat this process for all $i$. This gives us the skew tableau of shape $\la-\alpha$ with content $\beta$. This implies that $c^{\la}_{\alpha\beta}\geq 1$. Similarly, in the exact same method, we can show that $c^{\la}_{\alpha\gamma}$ and $c^{\la}_{\beta\gamma}$ are both at least $1$. All together, we conclude that $N^{\la}_{\la\la}\geq 1$, which completes the proof.
\end{proof}

The following example shows the skew tableau of shape $\la-\alpha$ with content $\beta$:
\begin{example}\label{ex1}
Let $\la=(6,4,4,2,2)$, $\alpha=(3,2,2,1,1)$, $\beta=(3, 2,2,1,1)$ be partitions as in the proof of Theorem \ref{thm-even}. The following Young diagrams explain how to obtain the skew-tableau of shape $\la-\alpha$ with content $\beta$ in the proof of Proposition \ref{prop-even}:
$$
\ytableausetup{aligntableaux=top}
\alpha=\ydiagram[*(yellow)]{3,2,2,1,1}\quad``+"\quad
\beta=\begin{ytableau} 1&1&1\\2&2\\3&3\\4\\5\end{ytableau}\quad\leadsto\quad
T=\ytableaushort
{\none \none \none  1 1 1, \none \none 22, \none\none 33, \none 4, \none 5}
*{6,4,4,2,2}
*[*(yellow)]{3,2,2, 1,1}
$$
Then $w(T)=(1,1,1,2,2,3,3,4,5)$ and $T$ is clearly a Littlewood-Richardson tableau.
\end{example}


\subsection{Even $\ell(\la)$ and all parts are distinct and odd} 

In this section, we consider the partitions $\la$ such that $\ell(\la)$ is even and all parts are distinct and odd.

\begin{prop}\label{prop-dist}
Let $\la$ be a partition such that $\ell(\la)$ even and all parts are distinct and odd. Then the representation $\otimes^3$ detects $\Sla(G)$.
\end{prop}

\begin{proof}
Let $\la$ be a partition such that $\ell(\la)=2k\leq n$ and all nonzero $\la_i$ are distinct and odd. As before, we need to construct the partitions $\alpha$, $\beta$ and $\gamma$ such that $c^{\la}_{\alpha \beta}\geq 1$, $c^{\la}_{\alpha \gamma}\geq 1$ and $c^{\la}_{\beta\gamma}\geq 1$. Let $\alpha$, $\beta$ and $\gamma$ be partitions such that
$$
\alpha=\gamma=(\tfrac{\la_1+1}{2},  \dots, \tfrac{\la_k+1}{2}, \tfrac{\la_{k+1}-1}{2},\dots, \tfrac{\la_{2k}-1}{2})
$$and
$$\beta=(\tfrac{\la_1-1}{2}, \dots, \tfrac{\la_k-1}{2}, \tfrac{\la_{k+1}+1}{2},\dots, \tfrac{\la_{2k}+1}{2}).
$$
Clearly $|\alpha|=|\beta|=|\gamma|=\tfrac{|\la|}{2}$. One notices that the assumption $\la_k>\la_{k+1}$ implies that, for $\beta$ to be a partition, the terms should satisfy the following inequality
$$
\tfrac{\la_k-1}{2}\geq \tfrac{\la_{k+1}+1}{2}.
$$
Moreover, by the assumption on $\la_i$ being distinct and odd, one notes that the last part $\tfrac{\la_{2k}-1}{2}$ of $\alpha$ and $\gamma$  could be zero if $\la_{2k}=1$.

Showing $c^{\la}_{\alpha\beta}\geq 1$ will be similar to the proof of Proposition \ref{prop-even}: For $1\leq i\leq k$, put $i$'s into $\tfrac{\la_i-1}{2}$ boxes in the $i$th row of $\beta$ and add them to the right side of the $\tfrac{\la_i+1}{2}$ empty boxes in the $i$th row of $\alpha$. This becomes then the $i$th row of $\la$ with $\la_i$ boxes where the first $\tfrac{\la_i+1}{2}$ boxes are empty and the remaining $\tfrac{\la_i-1}{2}$ boxes have $i$'s in them. For $k+1\leq i\leq 2k$, put $i$'s into $\tfrac{\la_i+1}{2}$ boxes in the $i$th row of $\beta$ and add them to the right side of the $\tfrac{\la_i-1}{2}$ empty boxes in the $i$th row of $\alpha$. This becomes then the $i$th row of $\la$ with $\la_i$ boxes where the first $\tfrac{\la_i-1}{2}$ boxes are empty and the remaining $\tfrac{\la_i+1}{2}$ boxes have $i$'s in them. If $\la_{2k}>1$, then the resulting tableau is just the skew tableau of shape $\la-\alpha$ with content $\beta$. If $\la_{2k}=1$, then the last part $\tfrac{\la_{2k}-1}{2}$ of $\alpha$ becomes zero, so $\alpha$ will have only $2k-1$ nonzero parts. In this case, we add the last box of $\beta$ containing $2k$ to the bottom of the first column of $\alpha$ (compare with the Young diagrams in the first case of Example \ref{ex2} below). This becomes then the last row of $\la$. This gives us the skew tableau of shape $\la-\alpha$ with content $\beta$. This implies that $c^{\la}_{\alpha\beta}\geq 1$ in either case. 

With the exact same method, we obtain the skew-tableau of shape $\la-\beta$ with content $\gamma$: For $1\leq i\leq k$, put $i$'s into $\tfrac{\la_i+1}{2}$ boxes in the $i$th row of $\gamma$ and add them to the right side of the $\tfrac{\la_i-1}{2}$ empty boxes in the $i$th row of $\beta$. This becomes then the $i$th row of $\la$ with $\la_i$ boxes where the first $\tfrac{\la_i-1}{2}$ boxes are empty and the remaining $\tfrac{\la_i+1}{2}$ boxes have $i$'s in them. For $k+1\leq i\leq 2k$, put $i$'s into $\tfrac{\la_i-1}{2}$ boxes in the $i$th row of $\gamma$ and add them to the right side of the $\tfrac{\la_i+1}{2}$ empty boxes in the $i$th row of $\beta$. This becomes then the $i$th row of $\la$ with $\la_i$ boxes where the first $\tfrac{\la_i+1}{2}$ boxes are empty and the remaining $\tfrac{\la_i-1}{2}$ boxes have $i$'s in them. If the last part $\tfrac{\la_{2k}-1}{2}$ of $\gamma$ is zero, we don't have any box to be added. In that case, the last row of $\beta$ will be just the last row of $\la$ (compare with the Young diagrams in the second case of Example \ref{ex2} below). This implies that $c^{\la}_{\beta\gamma}\geq 1$.

Now, in order to prove that $c^{\la}_{\alpha\gamma}\geq 1$, we will need to modify the above process slightly: For each $1\leq j\leq 2k$, put $j$'s into all boxes in the $j$th row of $\gamma$. To right of the first row of $\alpha$, add \textit{only} $\tfrac{\la_1-1}{2}$ boxes with $1$'s in them. This ensure $\la_1$ boxes in the first row of $\la$. The last box with $1$ in it should be added to right to the second row of $\alpha$. After that, we add the boxes with $2$ in them in the second row until we reach to $\la_2$ boxes all together. Whatever the remaining boxes with $2$ should be added into the third row. We repeat this process until we add all the boxes of $\gamma$ with numbers in them (compare with the Young diagrams in the third case of Example \ref{ex3} below).  In this way, we obtain the skew tableau of shape $\la-\alpha$ with content $\gamma$. Therefore we prove that $c^{\la}_{\alpha\gamma}\geq 1$ which completes the proof.
\end{proof}

The following example shows how to obtain skew tableau of shape $\la-\alpha$ with content $\beta$, skew tableau of shape $\la-\beta$ with content $\gamma$, and skew tableau of shape $\la-\alpha$ with content $\gamma$, respectively, by the process given in the proof of Proposition \ref{prop-dist}.

\begin{example}\label{ex2}
Let $\la=(7,5,3,1)$, $\alpha=(4,3,1,0)$, $\beta=(3,2,2,1)$ and $\gamma=(4,3,1,0)$ be partitions met the conditions in Proposition \ref{prop-dist}. The following Young diagrams explain how to obtain the skew-tableau of shape $\la-\alpha$ with content $\beta$ in the proof of Proposition \ref{prop-dist}:
$$
\ytableausetup{aligntableaux=top}
\alpha=\ydiagram[*(yellow)]{4,3,1}\quad``+"\quad
\beta=\begin{ytableau} 1&1&1\\2&2\\3&3\\4\end{ytableau}\quad\leadsto\quad
T_1=\ytableaushort
{\none \none \none\none 1 1 1, \none\none\none 22, \none33, 4 }
*{7,5,3,1}
*[*(yellow)]{4,3,1}
$$
Then $w(T_1)=(1,1,1,2,2,3,3,4)$ and $T_1$ is clearly a Littlewood-Richardson tableau. 

For the same $\la$, $\alpha$, $\gamma$ as above, the following Young diagrams explain the method for obtaining the skew-tableau of shape $\la-\beta$ with content $\gamma$:
$$
\ytableausetup{aligntableaux=top}
\beta=\ydiagram[*(yellow)]{3,2,2,1}\quad``+"\quad
\gamma=\begin{ytableau} 1&1&1&1\\2&2&2\\3\end{ytableau}\quad\leadsto\quad
T_2=\ytableaushort
{\none \none \none 1 1 1 1, \none\none 2 22, \none \none3, \none }
*{7,5,3,1}
*[*(yellow)]{3,2,2,1}
$$
Then $w(T_2)=(1,1,1,1,2,2,2,3)$ and $T_2$ is again a Littlewood-Richardson tableau. 

For the same $\la$, $\alpha$, $\gamma$ as above, the following Young diagrams explain the method for obtaining the skew-tableau of shape $\la-\alpha$ with content $\gamma$:
$$
\ytableausetup{aligntableaux=top}
\alpha=\ydiagram[*(yellow)]{4,3,1}\quad``+"\quad
\gamma=\begin{ytableau} 1&1 &1&1\\2&2&2\\3\end{ytableau}\quad\leadsto\quad
T_3=\ytableaushort
{\none \none \none \none 1 1 1, \none \none\none 1 2, \none 22, 3}
*{7,5,3,1}
*[*(yellow)]{4,3,1}
$$
Then $w(T_3)=(1,1,1,2,1,2,2,3)$ and $T_3$ is once again a Littlewood-Richardson tableau.
\end{example}

\subsection{Hook partitions}\label{hook}

In this section, we consider partitions $\la$ with at most $n$ parts whose shape is like a hook, namely
\begin{equation}\label{hook-la}
\la=(1+a, 1^b), \quad b\in \{0, 1, 2, \dots, n-1\}.
\end{equation}
Here $a$ is any nonnegative integer. We call such a partition $\la$ the hook partition with the arm length $a$ and the leg length $b$. Note that $|\la|=1+a+b$ and $\ell(\la)=b+1\leq n$. For example, the partition $\la=(1+3, 1^2)$ is a hook partition with arm length $3$ and the leg length $2$ and its Young diagram is 
$$
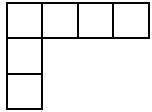
\yng(4,1,1)
$$
Clearly hook partitions are not fit to the cases considered in Proposition \ref{prop-even} and Proposition \ref{prop-dist}.

\begin{prop}\label{prop-hook}
Let $\la$ be a hook partition of an even number. Then the representation $\otimes^3$ detects $\Sla(G)$.
\end{prop} 

\begin{proof}
Fix $b\in \{0, 1, 2, \dots, n-1\}$ and let $\la=(1+a, 1^b)$ be the partition such that $1+a+b$ is even. Note that for $\mathrm{SO}(2n)$, $b$ ranges only up to $n-2$ due to the constraint on the $\la_n$ being zero.

Again, one needs to construct partitions $\alpha$, $\beta$ and $\gamma$ such that $c^{\la}_{\alpha\beta}\geq 1$, $c^{\la}_{\alpha\gamma}\geq 1$ and $c^{\la}_{\beta\gamma}\geq 1$, so is $N^{\la}_{\la\la}$ as well. Since $|\la|=1+a+b$ is even, we consider two separate cases: one is when $a$ is odd and $b$ is even and the other is when $a$ is even but $b$ is odd.

First assume that $a$ is odd and $b$ is even. Choose partitions $\alpha$, $\beta$ and $\gamma$ as 
$$
\alpha=\beta=\gamma=(1+\tfrac{a-1}{2}, 1^{b/2}).
$$
Then it is easy to obtain the skew tableau of shape $\la-\alpha$ with content $\beta$, for example: Put $1$'s into the $(1+\tfrac{a-1}{2})$ boxes in the first row of the Young diagram $\beta$ and add them to right side of the $(1+\tfrac{a-1}{2})$ empty boxes in the first row of $\alpha$. Then this becomes the first row of $\la$ with $1+a$ boxes, where the half $\tfrac{1+1}{2}$ boxes are empty and the remaining half $\tfrac{1+a}{2}$ boxes have $1$'s in them.
Likewise, put each $j$, $2\leq j\leq (\tfrac{b}{2}+1)$, into the each box in the leg of $\beta$. The the leg length $\tfrac{b}{2}$ of $\alpha$ can be extended to the leg length $b$ by adding $\tfrac{b}{2}$ boxes in the leg of $\beta$, where each box contains each $j$'s. This gives the skew tableau with shape $\la-\alpha$ with content $\beta$. Therefore $c^{\la}_{\alpha\beta}\geq 1$. Similar arguments show that the same is true for $c^{\la}_{\alpha\gamma}$ and $c^{\la}_{\beta\gamma}$. Hence we complete the proof for the first case.

Next, assume that $a$ is even and $b$ is odd. Choose the partition $\alpha$, $\beta$ and $\gamma$ as
$$
\alpha=\gamma=(1+\tfrac{a}{2}, 1^{(b-1)/2}) \quad\text{and}\quad \beta=(1+(\tfrac{a}{2}-1), 1^{(b+1)/2}).
$$
Then again, for instance, the arm length $\tfrac{a}{2}$ of $\alpha$ can be extended to the arm length $a$ by adding $\tfrac{a}{2}$ boxes in the first row of the Young diagram $\beta$, where each box contains $1$s. Likewise, the leg length $\tfrac{b-1}{2}$ of $\alpha$ can be extended to the leg length $b$ by adding $\tfrac{b+1}{2}$ boxes in the leg of $\beta$, where each box contains $j$'s, $2\leq j\leq(\tfrac{b+1}{2}+1)$. This gives the skew tableau of shape $\la-\alpha$ with content $\beta$. Therefore $c^{\la}_{\alpha\beta}\geq 1$. With the exact same method, one can show $c^{\la}_{\beta\gamma}\geq 1$. In this case, we will have $(1+\tfrac{a}{2})$ boxes with $1$'s in them instead.

In order to obtain skew-tableau of shape $\la-\alpha$ with content $\gamma$, we need a little more care: As before, put $1$'s into $1+\tfrac{a}{2}$ boxes in the first row of $\gamma$ and put $j$'s into each box in the leg of $\gamma$ for all $2\leq j\leq \tfrac{b+1}{2}$. Now add only $\tfrac{a}{2}$ boxes of $\gamma$ with $1$'s in them to the first row of $\alpha$ with $1+\tfrac{a}{2}$ empty boxes. The remaining one box containing $1$ should be placed to the end of the leg of $\alpha$. After that, the $\tfrac{b-1}{2}$ boxes in the leg of $\gamma$ are added to the leg.
This give us the first row of $\la$ with $1+a$ boxes where $1+\tfrac{a}{2}$ boxes are empty and $\tfrac{a}{2}$ boxes contain $1$. Also, the leg length is $b$, where first $\tfrac{b-1}{2}$ boxes are empty, one box contains $1$, and the rest $\tfrac{b-1}{2}$ boxes contain each $j$'s, $2\leq j\leq\tfrac{b+1}{2}$. This is the skew tableau of shape $\la-\alpha$ with content $\gamma$. Hence we complete the proof for the second case as well.
\end{proof}

Now we provide two examples to explain the proof above. Example \ref{ex3} is the first case and Example \ref{ex4}  is the second case of Proposition \ref{prop-hook}, respectively.

\begin{example}\label{ex3}
Let $\la=(1+5, 1^4)$ and let
$
\alpha=\beta=\gamma=(1+2, 1^2).
$
The following Young diagrams explain how to obtain the skew tableau of shape $\la-\alpha$ with content $\beta$ in the proof of Proposition \ref{prop-hook}:
$$
\ytableausetup{aligntableaux=top}
\alpha=\ydiagram[*(yellow)]{3,1,1}\quad``+"\quad\beta=\begin{ytableau} 1&1&1\\2\\3\end{ytableau}\quad\leadsto\quad 
T=\ytableaushort
{\none \none \none 1 1 1, \none,\none, 2, 3}
*{6,1,1,1,1}
*[*(yellow)]{3,1,1}
$$
Then $w(T)=(1,1,1,2,3)$ and $T$ is clearly a Littlewood-Richardson tableau.
\end{example}
\begin{example}\label{ex4}
Let $\la=(1+4, 1^3)$ and let
$$
\alpha=\gamma=(1+2, 1), \quad\text{and}\quad \beta=(1+1, 1^2).
$$
The following Young diagrams again explain how to obtain the skew tableau of shape $\la-\alpha$ with content $\beta$ in the proof of Proposition \ref{prop-hook} :
$$
\ytableausetup{aligntableaux=top}
\alpha=\ydiagram[*(yellow)]{3,1}\quad``+"\quad\beta=\begin{ytableau} 1&1\\2\\3\end{ytableau}\quad\leadsto \quad
T_1=\ytableaushort
{\none \none \none 1 1, \none, 2, 3}
*{5,1,1,1}
*[*(yellow)]{3,1}
$$
Then $w(T_1)=(1,1,2,3)$ and $T_1$ is clearly a Littlewood-Richardson tableau.

Similarly, the following Young diagrams explain how to obtain the skew tableau of shape $\la-\beta$ with content $\gamma$ in the proof of Proposition \ref{prop-hook} :
$$
\ytableausetup{aligntableaux=top}
\beta=\ydiagram[*(yellow)]{2,1,1}\quad``+"\quad\gamma=\begin{ytableau} 1&1&1\\2\end{ytableau}\quad\leadsto \quad
T_2=\ytableaushort
{\none \none 1 1 1, \none, \none, 2}
*{5,1,1,1}
*[*(yellow)]{2,1,1}
$$
Then $w(T_2)=(1,1,1,2)$ and $T_2$ becomes again a Littlewood-Richardson tableau.

The last case explains how to obtain the skew tableau of shape $\la-\alpha$ with content $\gamma$ in  the proof of Proposition \ref{prop-hook} :
$$
\ytableausetup{aligntableaux=top}
\alpha=\ydiagram[*(yellow)]{3,1}\quad``+"\quad\gamma=\begin{ytableau} 1&1&1\\2\end{ytableau}\quad\leadsto \quad
T_3=\ytableaushort
{\none \none \none 1 1, \none, 1, 2}
*{5,1,1,1}
*[*(yellow)]{3,1}
$$
Then $w(T_3)=(1,1,1,2)$ and $T_3$ is clearly a Littlewood-Richardson tableau.
\end{example}

\subsection{Rectangular partitions}\label{rec}

We say that $\la$ is a rectangular partition of size $m\times k$ if its Young diagram has exactly $m$ rows and each row has exactly $k$ boxes. In this section, we examine the size $m\times k$ rectangular partition $\la$, where $m\leq n$. Note that $|\la|=mk$ and $\ell(\la)=m$.  
 
\begin{prop}\label{prop-rec}
Let $\la$ be a rectangular partition of an even number. Then the representation $\otimes^3$ detects $\Sla(G)$.
\end{prop}
 
\begin{proof}
Let $\la$ be the rectangular partition of size $m\times k$ such that $|\la|=mk$ is even and $m\leq n$.

First assume that $k$ is even. Then it coincides with the case considered in Proposition \ref{prop-even}. Therefore one knows how to construct partitions $\alpha$, $\beta$ and $\gamma$ given by the process of the proof of Proposition \ref{prop-even} so that $c^{\la}_{\alpha\beta}\geq 1$, $c^{\la}_{\alpha\gamma}\geq 1$ and $c^{\la}_{\beta\gamma}\geq 1$ which in turn imply $N^{\la}_{\la\la}\geq 1$. This completes the proof in this case.

Next, assume that $m$ is even. Consider the conjugate partition $\la'$ of $\la$. Then $\la'$ becomes again the partition considered in Proposition \ref{prop-even}. Hence, for such a partition $\la'$, one can construct partitions $\alpha$, $\beta$ and $\gamma$ by the process of the proof of Proposition \ref{prop-even} so that $c^{\la'}_{\alpha\beta}\geq 1$, $c^{\la'}_{\alpha\gamma}\geq 1$ and $c^{\la'}_{\beta\gamma}\geq 1$. On the other hand, it is known that
$$
c^{\la}_{\mu\nu}=c^{\la'}_{\mu'\nu'}
$$see \cite{BSS} for example. Therefore, we can choose the corresponding conjugate partitions  $\alpha'$, $\beta'$ and $\gamma'$ such that $c^{\la}_{\alpha'\beta'}\geq 1$, $c^{\la}_{\alpha'\gamma'}\geq 1$ and $c^{\la}_{\beta'\gamma'}\geq 1$.  This implies the desired result.
\end{proof}

\section*{Acknowledgements}

The author is grateful to Leslie Saper for answering various questions on representation theory. The author also thanks to J. R. Getz for his constant support throughout this project and help with editing of the paper.



\begin{thebibliography}{}

\bibitem[An]{Andrews}
G. E. Andrews, \textbf{Number Theory}, Dover Publ., NY, 1971.


\bibitem[Ar]{Arthur}
J. Arthur, \textbf{The Endoscopic Classification of Representations: Orthogonal and Symplectic Groups}, Amer. Math. Soc.  Colloquium Publ., \textbf{61}, 2013.


\bibitem[BSS]{BSS}
G. Benkart, F. Sottile, J. Stroomer, \emph{Tableau switching: algorithms and applications}, J. Combin. Theory Ser. A \textbf{76} 91996), 11--34.


\bibitem[CKPSS]{CKPSS}
J. Cogdell, H. Kim, I. Piatetski-Shapiro, and F. Shahidi, \emph{Functoriality for classical groups}, Publ. Math. Inst. Hautes \'Etudes Sci. \textbf{99} (204), 163--233.



\bibitem[F]{Fulton}
W. Fulton, \emph{Young Tableaux}, London Math. Soc. Student Text \textbf{35}, Cambridge Univ. Press, Cambridge, 1997

\bibitem[FH]{FH}
W. Fulton and J. Harris, \emph{Representation theory}, Springer-Verlag, New York, 1991.



\bibitem[GK]{GK}
J. R. Getz and J. Klassen, \emph{Isolating Rankin-Selberg lifts}, Proc. Amer. Math. Soc., \textbf{143}, no. 8 (2015), 3319--3329.

\bibitem[GRS]{GRS}
 D. Ginzburg, S. Rallis and D. Soundry, \emph{Generic automorphic forms on $\mathrm{SO}(2n+1)$: Functorial lift to $\GL(2n)$}, \textbf{27} (2012), 143--211.

\bibitem[H]{Hahn}
H. Hahn, \emph{On tensor third $L$-functions of automorphic representations of $\GL_n(\A_F)$}, submitted for publication (arXiv:1509.01863).

\bibitem[HL]{HL}
R. Howe and S. T. Lee, \emph{Why should the Littlewood-Richardson rule be true?}, Bull. Amer. Math. Soc., \textbf{49}, no. 2 (2012), 187--236.

\bibitem[K]{King}
R. C. King, \emph{Modification rules and product of irreducible representations of the unitary, orthogonal and symplectic groups}, J. Math. Phys. \textbf{12} (1971), 1588--1598.


\bibitem[KT]{KT}
K. Koike and I. Terada, \emph{Young-diagrammatic methods for the representation theory of the classical groups of type $B_n$, $C_n$, $D_n$} J. Algebra \textbf{107}, No. 2, (1987), 466--511.

\bibitem[L1]{Langlands_conj}
R. P. Langlands, \emph{Letter to Andr\'e Weil}, 1967.

\bibitem[L2]{Langlands_beyond}
R. P. Langlands, \emph{Beyond endoscopy}, in \textbf{Contributions to Automorphic Forms, geometry, and Number Theory: A Volume in Honor of Joseph Shalika}, Johns Hopkins Univ. Press, 2004.





\bibitem[M]{Milne}
J. S. Milne, \emph{Algebraic Groups: An introduction to the theory of algebraic group schemes over fields}, www.jmilne.org/math/.




\bibitem[Se]{Seitz}
G. M. Seitz, \emph{Topics in the theory of algebraic groups}, in \textbf{Group Representation Theory}, edited by M. Geck, D. Testerman and J. Th\'evenaz, EPFL Press, 2007.


\bibitem[SZ]{SZ}
M. Shimozono and M. Zabrocki, \emph{Deformed universal characters for classical and affine algebras and the $X+M+K$ conjecture}, Formal power series and algebraic combinatorics, Vancouver 2004.





\end{thebibliography}
\end{document}